\documentclass[reqno]{amsart}
\usepackage{url}
\newtheorem{theorem}{Theorem}[section]
\newtheorem{lemma}[theorem]{Lemma}

\theoremstyle{definition}

\theoremstyle{remark}

\numberwithin{equation}{section}
\usepackage{algorithmic}
\usepackage[top=3cm,bottom=2cm,right=2cm,left=2cm]{geometry}

\begin{document}

\title[Divisibility of $7$-Elongated Plane Partition Diamonds by Powers of 8]{On the Divisibility of 7-Elongated Plane Partition Diamonds by Powers of 8}

%    Remove any unused author tags.

%    author one information
\author{James A. Sellers}
\address{}
\curraddr{}
\email{}
\thanks{}

%    author two information
\author{Nicolas Allen Smoot}
\address{}
\curraddr{}
\email{}
\thanks{}



\date{}

\dedicatory{}

\begin{abstract}
In 2021 da Silva, Hirschhorn, and Sellers studied a wide variety of congruences for the $k$-elongated plane partition function $d_k(n)$ by various primes.  They also conjectured the existence of an infinite congruence family modulo arbitrarily high powers of 2 for the function $d_7(n)$.  We prove that such a congruence family exists---indeed, for powers of 8.  The proof utilizes only classical methods, i.e., integer polynomial manipulations in a single function, in contrast to all other known infinite congruence families for $d_k(n)$ which require more modern methods to prove.
\end{abstract}

\maketitle

\section{Introduction}

The study of $k$-elongated plane partition diamonds began as a part of the series of papers on MacMahon's Partition Analysis by Andrews and Paule.  The number of $k$-elongated partition diamonds is denoted $d_k(n)$, and is enumerated by the function \begin{align}
D_k(q) := \sum_{n=0}^{\infty}d_k(n)q^n = \prod_{m=1}^{\infty}\frac{(1-q^{2m})^k}{(1-q^m)^{3k+1}}\label{D2}.
\end{align}  This function serves as a generalization of the unrestricted partition function $p(n) = d_0(n)$.  Moreover, $d_k(n)$ was developed by Andrews and Paule in \cite{AndrewsPaule2} as an application of the techniques of Partition Analysis.  In the latest paper in their series \cite{AndrewsPaule}, Andrews and Paule explored various congruence properties of $d_k(n)$, of the form of the classical congruences for $p(n)$ first studied by Ramanujan \cite{Ramanujan}.  In particular, they proposed \cite[Section 7]{AndrewsPaule} an infinite family of congruences for $d_2(n)$ modulo powers of 3, which was proved in \cite{Smoot2}.

More recently, da Silva, Hirschhorn, and Sellers \cite{dasilvaet} have explored a large variety of congruences by running both $n$ and $k$ through different arithmetic progressions.  Their work has revealed that $d_k(n)$ contains an enormous diversity of divisibility properties, even by the standards of partition theory.

As a result of this latest work, combined with numerical experimentation, da Silva, Hirschhorn, and Sellers conjectured \cite[Section 5]{dasilvaet} the existence of at least one infinite family of congruences over $d_7(n)$ by powers of 2.  We have confirmed this conjecture, as shown in the following result, which we prove in this paper:

\begin{theorem}\label{Thm12}
Let $n,\alpha\in\mathbb{Z}_{\ge 1}$ such that $3n\equiv 1\pmod{4^{\alpha}}$.  Then $d_7(n)\equiv 0\pmod{8^{\alpha}}$.
\end{theorem}

This congruence family is associated with the classical modular curve $\mathrm{X}_0(8)$.  The individual cases of the congruence family are enumerated by the coefficients of the modular function sequence $\mathcal{L} := \left( L_{\alpha} \right)_{\alpha\ge 1}$ defined by \begin{align}
L_{\alpha} := \prod_{m=1}^{\infty}\frac{(1-q^m)^{26}(1-q^{4m})^{2}}{(1-q^{2m})^{13}}\cdot\sum_{n=0}^{\infty}d_k(4^{\alpha}n+\lambda_{\alpha})q^{n+1},\label{lalphadefn}
\end{align} with $q := e^{2\pi i\tau}$, $\tau\in\mathbb{H}$, and $y=\lambda_{\alpha}$ the minimum positive solution to $3y\equiv 1\pmod{4^{\alpha}}$.

The curve $\mathrm{X}_0(8)$ has genus 0, and cusp count greater than 2.  Ordinarily this would make the classical techniques difficult to apply.  Indeed, the congruence family for $d_2$ modulo powers of 3 proved in \cite{Smoot2} requires more modern techniques, notably the localization method.

However, in the case of Theorem \ref{Thm12} we can construct an associated sequence of modular functions $L_{\alpha}$ which live at a single cusp.  This allows a polynomial representation of each $L_{\alpha}$ in terms of a Hauptmodul.

As an example, we consider the case $\alpha=1$: \begin{align*}
L_1&=\prod_{m=1}^{\infty}\frac{(1-q^m)^{26}(1-q^{4m})^{2}}{(1-q^{2m})^{13}}\cdot\sum_{n=0}^{\infty}d_7(4n+3)q^{n+1}.
\end{align*}  If we express $L_1$ in terms of a certain Hauptmodul $x$ (defined in (\ref{xdefn}) below) which expands to an integer-power series (to be defined in the following section), we have the following: \begin{align}
L_1 &= 2376 x + 2769184 x^2 + 753360896 x^3 + 87754260480 x^4 + 
 5608324988928 x^5 + 224018944753664 x^6\label{L1inx}\\ &+ 6042206699782144 x^7 + 
 115546340691279872 x^8 + 1616547968486211584 x^9\nonumber\\ &+ 
 16870983657286795264 x^{10} + 132703559201308278784 x^{11} + 
 788474037948865576960 x^{12}\nonumber\\ &+ 3517455424164433231872 x^{13} + 
 11593058074386073911296 x^{14} + 27374968205384974598144 x^{15}\nonumber\\ &+ 
 43792570430986475536384 x^{16} + 42501298345826806923264 x^{17} + 
 18889465931478580854784 x^{18}\nonumber.
\end{align}  A quick examination of the coefficients of (\ref{L1inx}) demonstrates that \begin{align*}
L_1 &\equiv 0\pmod{8}.
\end{align*} Using (\ref{L1inx}) in combination with some techniques in the theory of modular functions, we prove Theorem \ref{Thm12} as a corollary of the following theorem:

\begin{theorem}\label{Thmmaintoc}
For all $\alpha\ge 1$, \begin{align}
\frac{1}{8^{\alpha}}L_{\alpha}\in\mathbb{Z}[x].
\end{align}
\end{theorem}

The remainder of this paper is as follows: In Section \ref{setupsec} we define the key linear operator $U$ which maps each $L_{\alpha}$ to $L_{\alpha+1}$.  We then define our Hauptmodul $x$, list some of its key properties, and give a modular equation (\ref{modX}) from which we can build useful recurrence relations for $U(x^n)$.

In Section \ref{mainlemmasection} we show how to represent $U(x^n)$ as a polynomial in $x$, paying careful attention to the 2-adic valuation of the coefficients of $x^n$.  The main result in this section, Lemma \ref{mainlemmag}, is proved by an induction which requires four initial relations be proved directly.

In Section \ref{maintheoremsection} we construct the polynomial space $\mathcal{V}$ in which our functions $L_{\alpha}$ live.  We show that applying our operator $U$ to elements of $\mathcal{V}$ results in new elements of $\mathcal{V}$ in which each coefficient gains an additional multiplicative factor of 8.  This is sufficient for us to complete the proof of Theorems \ref{Thmmaintoc} and \ref{Thm12}.

In Section \ref{initialrelsection} we show how to prove the initial relations for Lemma \ref{mainlemmag}.  The relations themselves are included in the Appendix.  We also verify the identity (\ref{L1inx}) for $L_1$, and the modular equation (\ref{modX}).

In Section \ref{additionalsection} we briefly discuss the prospect of finding more congruence families for $d_k(n)$, and where more work is needed.

\section{Setup}\label{setupsec}

For the sake of convenience, we define the standard $q$-Pochhammer symbol by

\begin{align*}
(q^a;q^b)_{\infty} := \prod_{m=0}^{\infty} \left( 1-q^{a+bm} \right).
\end{align*}  Notice that with this notation, we have \begin{align*}
L_{\alpha} := \frac{(q;q)^{26}_{\infty}(q^4;q^4)^{2}_{\infty}}{(q^2;q^2)^{13}_{\infty}}\sum_{n=0}^{\infty}d_k(4^{\alpha}n+\lambda_{\alpha})q^{n+1}.
\end{align*}

It is a straightforward process to prove that each $L_{\alpha}$ as defined in (\ref{lalphadefn}) is a modular function over the congruence subgroup $\Gamma_0(8)$, as we will show in Section \ref{initialrelsection}.  It is only slightly more difficult to form the necessary linear operator that permits us to construct $L_{\alpha+1}$ from $L_{\alpha}$.  With this in mind, we define our key operator, \begin{align}
U\left( f \right) := U_4\left( \mathcal{A}\cdot f \right),\label{uoperatordefn}
\end{align} with \begin{align}
\mathcal{A} := q\frac{(q^2;q^2)_{\infty}^{13}(q^4;q^4)_{\infty}^{24}(q^{16};q^{16})_{\infty}^2}{(q;q)_{\infty}^{26}(q^8;q^8)_{\infty}^{13}}.
\end{align}  We can quickly verify that $U$ is a linear operator.

\begin{lemma}
For all $\alpha\ge 1$, $U\left(L_{\alpha}\right) = L_{\alpha+1}$.
\end{lemma}

\begin{proof}

We use the important property that, for any power series $f(q) = \sum_{n\ge N} a(n)q^n$, $g(q) = \sum_{n\ge M} b(n)q^n$ \begin{align}
U_4\left( f(q^4)g(q) \right) = f(q)\cdot U_4\left( g(q) \right).
\end{align}  Now, applying $U$ to $L_{\alpha}$, we have \begin{align*}
U\left( L_{\alpha} \right) &= U_4\left(\mathcal{A}\cdot L_{\alpha} \right)\\
&= U_4\left(\frac{(q^2;q^2)_{\infty}^{13}(q^4;q^4)_{\infty}^{24}(q^{16};q^{16})_{\infty}^2}{(q;q)_{\infty}^{26}(q^8;q^8)_{\infty}^{13}}\cdot \frac{(q;q)^{26}_{\infty}(q^4;q^4)^{2}_{\infty}}{(q^2;q^2)^{13}_{\infty}}\sum_{n=0}^{\infty}d_k\left(4^{\alpha}n+\lambda_{\alpha}\right)q^{n+2}\right)\\
&= U_4\left(\frac{(q^4;q^4)^{26}_{\infty}(q^{16};q^{16})^{2}_{\infty}}{(q^8;q^8)^{13}_{\infty}}\cdot\sum_{n=0}^{\infty}d_k\left(4^{\alpha}n+\lambda_{\alpha}\right)q^{n+2}\right)\\
&= \frac{(q;q)^{26}_{\infty}(q^{4};q^{4})^{2}_{\infty}}{(q^2;q^2)^{13}_{\infty}} U_4\left(\sum_{n\ge 2}d_k\left(4^{\alpha}(n-2)+\lambda_{\alpha}\right)q^{n}\right)\\
&= \frac{(q;q)^{26}_{\infty}(q^{4};q^{4})^{2}_{\infty}}{(q^2;q^2)^{13}_{\infty}} \sum_{4n\ge 2} d_k\left(4^{\alpha}(4n-2)+\lambda_{\alpha}\right)q^{n}
\end{align*}  Notice that $4n\ge 2$ is equivalent to $n\ge 1$.  Moreover, we can show that \begin{align}
\lambda_{\alpha} = \frac{4^{\alpha}\cdot 2 + 1}{3}.
\end{align}  We therefore have \begin{align*}
U\left( L_{\alpha} \right) &= \frac{(q;q)^{26}_{\infty}(q^{4};q^{4})^{2}_{\infty}}{(q^2;q^2)^{13}_{\infty}} \sum_{n\ge 1} d_k\left(4^{\alpha+1}n - 2\cdot 4^{\alpha} +\frac{4^{\alpha}\cdot 2 + 1}{3}\right)q^{n}\\
&= \frac{(q;q)^{26}_{\infty}(q^{4};q^{4})^{2}_{\infty}}{(q^2;q^2)^{13}_{\infty}} \sum_{n\ge 1} d_k\left(4^{\alpha+1}n +\frac{4^{\alpha}\cdot 2 - 3\cdot 4^{\alpha}\cdot 2 + 1}{3}\right)q^{n}\\
&= \frac{(q;q)^{26}_{\infty}(q^{4};q^{4})^{2}_{\infty}}{(q^2;q^2)^{13}_{\infty}} \sum_{n=0}^{\infty} d_k\left(4^{\alpha+1}n + 4^{\alpha+1} +\frac{4^{\alpha}\cdot 2 - 3\cdot 4^{\alpha}\cdot 2 + 1}{3}\right)q^{n+1}\\
&= \frac{(q;q)^{26}_{\infty}(q^{4};q^{4})^{2}_{\infty}}{(q^2;q^2)^{13}_{\infty}} \sum_{n=0}^{\infty} d_k\left(4^{\alpha+1}n +\frac{4^{\alpha+1}\cdot 2 + 1}{3}\right)q^{n+1}\\
&= L_{\alpha+1}.
\end{align*}

\end{proof}

We now have our sequence $\mathcal{L}$ and our operator $U$ properly defined.  Now it remains to express each $L_{\alpha}$ in terms of some reference function over which we have greater control.  We can compute the monoid of eta quotients over $\Gamma_0(8)$ which live at a single cusp of $\mathrm{X}_0(8)$, say $[0]_8$.  We find a suitable Hauptmodul in the form \begin{align}
x := q\frac{(q^2;q^2)^2_{\infty}(q^8;q^8)^4_{\infty}}{(q;q)^4_{\infty}(q^4;q^4)^2_{\infty}}.\label{xdefn}
\end{align}  We can verify that $x$ is a modular function with respect to $\Gamma_0(8)$ by using \cite[Theorem 1]{Newman}, and we can compute the orders of $x$ with respect to $\Gamma_0(8)$ quickly \cite[Theorem 23]{Radu}: \begin{align*}
\mathrm{ord}_{\infty}(x) &= 1,\\
\mathrm{ord}_{1/4}(x) &= 0,\\
\mathrm{ord}_{1/2}(x) &= 0,\\
\mathrm{ord}_{0}(x) &= -1.
\end{align*}  We can also compute a lower bound for the orders of $L_{1}$, and find that $\mathrm{ord}_{c}\left(L_{1}\right) \ge 0$ except for $c\in [0]_8$.  We should therefore be able to represent $L_{1}$ in terms of powers of $x$ (indeed, we have given the representation above in (\ref{L1inx})).  We will need to examine how $U$ affects powers of $x$.  But before we do, we need to prepare an important result for the function $x$:

\begin{theorem}\label{modeqntheorem}
Define \begin{align*}
a_0(\tau) &= -(x + 20 x^2 + 128 x^3 + 256 x^4)\\
a_1(\tau) &=-(16 x + 320 x^2 + 2048 x^3 + 4096 x^4)\\
a_2(\tau) &= -(80 x + 1600 x^2 + 10240 x^3 + 20480 x^4)\\
a_3(\tau) &=-(128 x + 2560 x^2 + 16384 x^3 + 32768 x^4).
\end{align*}  Then \begin{align}
x^4+\sum_{j=0}^3 a_j(4\tau) x^j = 0.\label{modX}
\end{align}
\end{theorem}

This is the modular equation in $x$ which will permit us to build useful recurrence relations for $U(x^n)$.  This may be constructed and proved using the cusp analysis which is standard to the subject of modular functions.

Notice that we can define integers $a(j,k)$ such that \begin{align}
a_j(\tau) = \sum_{k=1}^4 a(j,k) 2^{\phi(j,k)}x^k,\label{expressionaj}
\end{align} with \begin{align}
\phi(j,k) &:= \begin{cases} 
   0, &  j=0, k=1 \\
   2, &  j=0, k=2 \\
   7, &  j=0, k=3 \\
   8, &  j=0, k=4 \\
   \left\lfloor \frac{4k+2j+1}{2} \right\rfloor, &  1\le j\le 3.
  \end{cases}\label{expressionajk}
\end{align}  This function is a useful lower bound for the 2-adic valuation of the coefficients of $a_j(\tau)$, as we will see below.

\section{Main Lemma}\label{mainlemmasection}

We have seen that for each $L_{\alpha}$ we can obtain $L_{\alpha+1}$ by application of our operator $U$.  Because we intend to represent each $L_{\alpha}$ in terms of $x$, we next need to study how $U$ affects powers of $x$.  With this in mind, we define the function: \begin{align}
\pi(n,r) &:= \begin{cases} 
   \left\lfloor \frac{8r-5}{4} \right\rfloor + 3, &  n=1, \\
   \left\lfloor \frac{8r-5}{4} \right\rfloor + 1, & n=2, \\
   \left\lfloor \frac{4r-n-1}{2} \right\rfloor, & n\ge 3.
  \end{cases}
\end{align}  This will serve as a lower bound to the 2-adic valuation of the coefficients of $x^r$ in $U(x^n)$.

\begin{lemma}\label{mainlemmag}
For all $n\ge 1$, there exists a discrete array $h(n,r)$ such that \begin{align}
U\left( x^n \right) = \sum_{r\ge \left\lceil (n+1)/4 \right\rceil} h(n,r)\cdot 2^{\pi(n,r)}x^r.\label{uxnform}
\end{align}
\end{lemma}

\begin{proof}
We compute $U\left( x^n \right)$ for $0\le n\le 3$ directly, and give the polynomials in the Appendix below.

Let us take some $n\ge 4$, and suppose that (\ref{uxnform}) applies to $U(x^m)$ for $0\le m\le n$.  We begin by rearranging (\ref{modX}), and then multiplying by $\mathcal{A}\cdot x^{n-4}$: \begin{align*}
x^4 &= -\sum_{j=0}^3 a_j(4\tau) x^j,\\
x^n &= -\sum_{j=0}^3 a_j(4\tau) x^{n+j-4},\\
\mathcal{A}\cdot x^n &= -\sum_{j=0}^3 a_j(4\tau) \mathcal{A}\cdot x^{n+j-4}.
\end{align*}  We now apply our $U$ operator from (\ref{uoperatordefn}) to both sides: \begin{align*}
U_4\left(\mathcal{A}\cdot x^n\right) &= -\sum_{j=0}^3 a_j(\tau) U_4\left(\mathcal{A}\cdot x^{n+j-4}\right),\\
U\left(x^n\right) &= -\sum_{j=0}^3 a_j(\tau) U\left(x^{n+j-4}\right).
\end{align*}  By hypothesis, (\ref{uxnform}) applies for $U\left(x^{n+j-4}\right)$ for $0\le j\le 3$.  We therefore have \begin{align*}
U\left(x^n\right) &= -\sum_{j=0}^3 a_j(\tau) \sum_{r\ge \left\lceil (n+j-3)/4 \right\rceil} h(n+j-4,r)\cdot 2^{\pi(n+j-4,r)}x^r.
\end{align*}  Remembering (\ref{expressionaj})-(\ref{expressionajk}), we have \begin{align*}
U\left(x^n\right) &= -\sum_{j=0}^3 \sum_{k=1}^4 a(j,k) 2^{\phi(j,k)}x^k \sum_{r\ge \left\lceil (n+j-3)/4 \right\rceil} h(n+j-4,r)\cdot 2^{\pi(n+j-4,r)}x^r\\
U\left(x^n\right) &= -\sum_{\substack{0\le j\le 3\\ 1\le k\le 4\\ r\ge \left\lceil (n+j-3)/4 \right\rceil}} a(j,k)h(n+j-4,r)\cdot 2^{\pi(n+j-4,r)+\phi(j,k)}x^{r+k}.
\end{align*}  If we relabel our powers of $x$, we can rewrite our relation as \begin{align}
U\left(x^n\right) &= -\sum_{\substack{0\le j\le 3\\ 1\le k\le 4\\ r-k\ge \left\lceil (n+j-3)/4 \right\rceil}} a(j,k)h(n+j-4,r-k)\cdot 2^{\pi(n+j-4,r-k)+\phi(j,k)}x^{r}.\label{completelemmainduc}
\end{align}  Notice that \begin{align}
r\ge \left\lceil (n+j-3)/4 \right\rceil + k \ge \left\lceil (n-3)/4 \right\rceil + 1 = \left\lceil (n+1)/4 \right\rceil.
\end{align}  We now verify that $\pi(n+j-4,r-k)+\phi(j,k)\ge \pi(n,r)$.  For $j=0$, \begin{align*}
\pi(n-4,r-1)+\phi(1,1) &= \left\lfloor \frac{4r-n-1}{2} \right\rfloor = \pi(n,r),\\
\pi(n-4,r-2)+\phi(1,2) &= \left\lfloor \frac{4r-n-5}{2} \right\rfloor + 2 = \left\lfloor \frac{4r-n-1}{2} \right\rfloor = \pi(n,r),\\
\pi(n-4,r-3)+\phi(1,3) &= \left\lfloor \frac{4r-n-9}{2} \right\rfloor + 7 = \left\lfloor \frac{4r-n+5}{2} \right\rfloor\ge \pi(n,r)\\
\pi(n-4,r-4)+\phi(1,4) &= \left\lfloor \frac{4r-n-13}{2} \right\rfloor + 8 = \left\lfloor \frac{4r-n+3}{2} \right\rfloor\ge \pi(n,r).
\end{align*}  For $1\le j\le 3$, \begin{align*}
\pi(n+j-4,r-k)+\phi(j,k) &= \left\lfloor \frac{4r-4k-n-j+3}{2} \right\rfloor + \left\lfloor \frac{4k+2j+1}{2} \right\rfloor\\
&\ge \left\lfloor \frac{4r-n+j+3}{2} \right\rfloor\ge\pi(n,r).
\end{align*}  For a given $r\ge 0$, we can now define the value of our discrete array $h(n,r)$ as the coefficient of $x^r$ in (\ref{completelemmainduc}) after dividing out by $2^{\pi(n,r)}$.  We only need to know that $h(n,r)$ exists for $0\le n\le 3$.  This can be shown by computing the four initial cases $U(x^n)$ directly for $0\le n\le 3$, confirming that it is a polynomial in $x$, and and showing that each coefficient of $x^r$ is divisible by $2^{\pi(n,r)}$.  We provide the explicit forms of these initial cases in the Appendix.

\end{proof}

\section{Main Theorem}\label{maintheoremsection}

In order to prove Theorem \ref{Thmmaintoc}, we will show that each $L_{\alpha}$ is a polynomial in $x$ in which each coefficient of $x$ is divisible by $8^{\alpha}$.  We already understand that this is true for $\alpha=1$, by (\ref{L1inx}).  We therefore need to construct a subspace of integer polynomials in $x$ which includes $L_{1}$, and which is stable upon application of the operator $U$.  We define a subspace of integer polynomials in $x$ with a certain lower bound on the 2-adic valuation of the coefficients of $x$.  For \begin{align*}
\theta(n) &:= \left\lfloor \frac{8n-5}{4} \right\rfloor,
\end{align*} define the space \begin{align}
\mathcal{V} :=& \left\{ \sum_{n\ge 1} s(n)\cdot 2^{\theta(n)}\cdot x^n \right\},
\end{align} in which $s(n)$ is any discrete integer-valued function.

\begin{theorem}\label{bigtheorem}
For all $f\in\mathcal{V}$, \begin{align}
\frac{1}{8}U\left( f \right)\in\mathcal{V}.
\end{align}
\end{theorem}  With this result we prove Theorem \ref{Thmmaintoc}.

\begin{proof}
Let $f\in\mathcal{V}$.  Then \begin{align}
f = \sum_{n\ge 1} s(n)\cdot 2^{\theta(n)}\cdot x^n.
\end{align}  Applying $U$, we have \begin{align}
U\left(f\right) &= \sum_{n\ge 1} s(n)\cdot 2^{\theta(n)}U\left(x^n\right)\\
&= \sum_{n\ge 1}\sum_{r\ge \left\lceil (n+1)/4 \right\rceil} s(n) h(n,r)\cdot 2^{\pi(n,r)+\theta(n)}x^r\\
&= \sum_{r\ge 1}\sum_{n\ge 1} s(n) h(n,r)\cdot 2^{\pi(n,r)+\theta(n)}x^r.
\end{align}  We want to show that \begin{align}
\pi(n,r)+\theta(n)\ge\theta(r)+3
\end{align} for all $n\ge 1$, $r\ge \left\lceil (n+1)/4 \right\rceil$.  In that case, it will be demonstrated that $U\left(f\right)$ remains in $\mathcal{V}$, and that each factor gains a power of 8, whereupon we have Theorem \ref{bigtheorem}.

First examining the cases for $n=1,2$, we have \begin{align}
\pi(1,r)+\theta(1) &= \left( \left\lfloor \frac{8r-5}{4} \right\rfloor + 3\right) + 0 = \theta(r) + 3,\\
\pi(2,r)+\theta(2) &= \left(\left\lfloor \frac{8r-5}{4} \right\rfloor + 1\right) + 2 = \theta(r) + 3.
\end{align}  For $n\ge 3$ we have \begin{align}
\pi(n,r)+\theta(n) &= \left\lfloor \frac{4r-n-1}{2} \right\rfloor + \left\lfloor\frac{8n-5}{4}\right\rfloor\\
&= \left\lfloor \frac{8r-2n-2}{4} \right\rfloor + \left\lfloor\frac{8n-5}{4}\right\rfloor\\
&\ge \left\lfloor \frac{8r+6n-10}{4} \right\rfloor\\
&\ge \left\lfloor \frac{8r-5}{4} \right\rfloor + \left\lfloor \frac{6n-5}{4} \right\rfloor\\
&\ge \theta(r) + 3.
\end{align}  

\end{proof}

\begin{proof}[Proof of Theorem \ref{Thmmaintoc}]

Notice that by (\ref{L1inx}), $L_{1}\equiv 0\pmod{8}$, and $\frac{1}{8}L_1\in\mathcal{V}$.  Suppose that for some fixed $\alpha\ge 1$, $L_{\alpha}\equiv 0\pmod{8^{\alpha}}$, and that $\frac{1}{8^{\alpha}}L_{\alpha}\in\mathcal{V}$.  Then \begin{align}
\frac{1}{8}\cdot U\left( \frac{1}{8^{\alpha}}L_{\alpha} \right) &= \frac{1}{8^{\alpha+1}}\cdot U\left(L_{\alpha} \right) \\
&= \frac{1}{8^{\alpha+1}}\cdot L_{\alpha+1}\in\mathcal{V}\subseteq\mathbb{Z}[x],
\end{align} whence $L_{\alpha+1}\equiv 0\pmod{8^{\alpha+1}}$.
\end{proof}  From this, Theorem \ref{Thm12} immediately follows.

\section{Proving the Initial Relations}\label{initialrelsection}

Some important aspects of our proof which we have not yet covered are the validity of (\ref{L1inx}), the modular equation in Theorem \ref{modeqntheorem}, and the relations for $U(x^n)$, $0\le n\le 3$ which appear in (\ref{relationapu1})-(\ref{relationapux3}) in the Appendix below.

These results may be proved using the now standard techniques of the modular cusp analysis.  For a classical approach to the theory, see \cite{Knopp}.  For a more modern approach, see \cite{Diamond}.  For a description of the theory specifically applied to partition congruences, see \cite[Section 2]{Smoot2}.  However, we will focus on an algorithmic approach, because it provides us with the most convenient way of proving our final steps.  We therefore defer to Newman's important paper \cite{Newman}, Radu's algorithmic procedure in \cite{Radu}, and its implementation in \cite{Smoot1}.

Beginning with the case of (\ref{L1inx}), we define \begin{align*}
\texttt{xq} = 1/x.
\end{align*}  We then use the RaduRK package to compute \begin{flalign*}
\texttt{In[1] = }&\texttt{RKMan} [ 8,2,\{-22,7\},4,3,\{1/\texttt{xq},\{1\}\} ]&\\
&\tt{\prod_{\delta | M} (q^{\delta};q^{\delta})^{r_{\delta}}_{\infty} = \sum_{n=0}^{\infty}a(n)q^n}\\
&\boxed{\tt{f_1(q)\cdot\prod_{j'\in P_{m,r}(j)}\sum_{n=0}^{\infty}a(mn+j')q^n = \sum_{g\in\mathrm{AB}} g\cdot p_g(t)}}\\
&\texttt{Modular Curve: }\texttt{X}_{\texttt{0}}\texttt{(N)}\\
\texttt{Out[1] = }&\\
&\texttt{N:}\ \ \ \ \ \ \ \ \ \ \ \ \ \ \ \ \ \ \ \ \ \ \ \ \ \ \ \ \ \ 8\\
&\texttt{\{M, (}\texttt{r}_{\tt{\delta | M}}\texttt{)\}}\texttt{:}\ \ \ \ \ \ \ \ \ \ \ \  \{2, \{-22,7\}\} \\
&\texttt{m :} \ \ \ \ \ \ \ \ \ \ \ \ \ \ \ \ \ \ \ \ \ \ \ \ \ \ \ \ 4\\
&\texttt{P}_{\texttt{m,r}}\texttt{(j): } \ \ \ \ \ \ \ \ \ \ \ \ \ \ \ \  \{3\}\\
&\texttt{f}_{\texttt{1}}\texttt{(q): } \ \ \ \ \ \ \ \ \ \ \ \ \ \ \ \ \ \ \  \frac{((q;q)_{\infty})^{98}((q^4;q^4)_{\infty})^{38}}{q^{17}((q^2;q^2)_{\infty})^{49}((q^8;q^8)_{\infty})^{72}}\\
&\texttt{t: }  \ \ \ \ \ \ \ \ \ \ \ \ \ \ \ \ \ \ \ \ \ \ \ \ \ \frac{((q;q)_{\infty})^4((q^4;q^4)_{\infty})^4}{q((q^2;q^2)_{\infty}^2((q^8;q^8)_{\infty})^4}\\
&\texttt{AB: } \ \ \ \ \ \ \ \ \ \ \ \ \ \ \ \ \ \ \ \ \ \ \ \ \{1\}\\
&\{ \texttt{p}_{\texttt{g}}\texttt{(t):g}\in\texttt{AB}\}\texttt{: }  \ \ \ \ \ \{18889465931478580854784 + 42501298345826806923264 t\\ 
& \ \ \ \ \ \ \ \ \ \ \ \ \ \ \ \ \ \ \ \ \ \ \ \ \ \ \ \ \ \ \ +43792570430986475536384 t^2 + 27374968205384974598144 t^3\\ 
& \ \ \ \ \ \ \ \ \ \ \ \ \ \ \ \ \ \ \ \ \ \ \ \ \ \ \ \ \ \ \ + 
 11593058074386073911296 t^4 + 3517455424164433231872 t^5\\ 
& \ \ \ \ \ \ \ \ \ \ \ \ \ \ \ \ \ \ \ \ \ \ \ \ \ \ \ \ \ \ \ + 
 788474037948865576960 t^6 + 132703559201308278784 t^7\\ 
& \ \ \ \ \ \ \ \ \ \ \ \ \ \ \ \ \ \ \ \ \ \ \ \ \ \ \ \ \ \ \ + 
 16870983657286795264 t^8 + 1616547968486211584 t^9\\ 
& \ \ \ \ \ \ \ \ \ \ \ \ \ \ \ \ \ \ \ \ \ \ \ \ \ \ \ \ \ \ \ + 
 115546340691279872 t^{10} + 6042206699782144 t^{11}\\ 
& \ \ \ \ \ \ \ \ \ \ \ \ \ \ \ \ \ \ \ \ \ \ \ \ \ \ \ \ \ \ \ + 
 224018944753664 t^{12} + 5608324988928 t^{13} + 87754260480 t^{14}\\ 
& \ \ \ \ \ \ \ \ \ \ \ \ \ \ \ \ \ \ \ \ \ \ \ \ \ \ \ \ \ \ \ + 
 753360896 t^{15} + 2769184 t^{16} + 2376 t^{17}\}\\
&\texttt{Common Factor: } \ \ \ \ \ \ \  8
\end{flalign*} The critical point here is that the program produces a function which is a modular function with a pole only at the cusp $[\infty]_{8}$: \begin{align*}
\frac{(q;q)_{\infty}^{98}(q^4;q^4)_{\infty}^{38}}{q^{17}(q^2;q^2)_{\infty}^{49}(q^8;q^8)_{\infty}^{72}}\cdot\sum_{n=0}^{\infty}d_7(4n+3)q^n\in\mathcal{M}^{\infty}\left(  \Gamma_0(8)\right).
\end{align*}  This function is then shown to be a polynomial in $t$, which is in our case $1/x$.  If we multiply this function by $x^{18}$, we have \begin{align*}
&x^{18}\cdot\left(\frac{(q;q)_{\infty}^{98}(q^4;q^4)_{\infty}^{38}}{q^{17}(q^2;q^2)_{\infty}^{49}(q^8;q^8)_{\infty}^{72}}\cdot\sum_{n=0}^{\infty}d_7(4n+3)q^n\right)\\
&=\frac{(q;q)^{26}_{\infty}(q^4;q^4)^{2}_{\infty}}{(q^2;q^2)^{13}_{\infty}}\cdot\sum_{n=0}^{\infty}d_7(4n+3)q^{n+1}\\
&= L_1.
\end{align*}  We then multiply our polynomial in $t$ by $x^{18}$, and express $t=1/x$.  Doing so gives us (\ref{L1inx}).  Relations (\ref{relationapu1})-(\ref{relationapux3}) may be similarly computed.

The modular equation may be derived using (\ref{relationapu1})-(\ref{relationapux3}), but we will only give its proof here.  We note that $x(4\tau),x(\tau)$ are both modular functions over the subgroup $\Gamma_0(32)$, \cite[Theorem 1]{Newman}.  We can compute the orders of these functions with respect to this subgroup using \cite[Theorem 23]{Radu}.  Doing so gives us: \begin{align*}
\mathrm{ord}_{\infty}(x(4\tau)) &= -4,\ \ \ \ \ \ \ \mathrm{ord}_{\infty}(x) = 1, \\
\mathrm{ord}_{1/16}(x(4\tau)) &= 0,\ \ \ \ \ \ \ \mathrm{ord}_{1/16}(x) = 1,\\
\mathrm{ord}_{1/8}(x(4\tau)) &= 0,\ \ \ \ \ \ \ \mathrm{ord}_{1/8}(x) = 1,\\
\mathrm{ord}_{1/4}(x(4\tau)) &= 0,\ \ \ \ \ \ \ \mathrm{ord}_{1/4}(x) = 0,\\
\mathrm{ord}_{3/8}(x(4\tau)) &= 0,\ \ \ \ \ \ \ \mathrm{ord}_{3/8}(x) = 1,\\
\mathrm{ord}_{1/2}(x(4\tau)) &= 1,\ \ \ \ \ \ \ \mathrm{ord}_{1/2}(x) = 0,\\
\mathrm{ord}_{3/4}(x(4\tau)) &= 1,\ \ \ \ \ \ \ \mathrm{ord}_{3/4}(x) = 0,\\
\mathrm{ord}_{0}(x(4\tau)) &= 1,\ \ \ \ \ \ \ \mathrm{ord}_{0}(x) = -4.
\end{align*}  Notice that $x(4\tau)^{-4}x(\tau)\in\mathcal{M}^{\infty}\left(  \Gamma_0(8)\right)$.  So if we divide the left-hand side of (\ref{modX}) by $x(4\tau)^{16}$, then the left-hand side is converted into a modular function with a pole only at the cusp $[\infty]_{32}$.

Therefore, we need only check the principal part and constant of the power series expansion of \begin{align}
x(4\tau)^{-16}\left(x^4+\sum_{j=0}^3 a_j(4\tau) x^j\right).
\end{align}  This can be quickly computed with any computer algebra system (or, indeed, by hand) as 0.  We need not compute any additional coefficients.

\section{Additional Congruence Results}\label{additionalsection}

The natural question that follows is whether other infinite congruence families modulo powers of 2 exist for $d_k$ over other values of $k$.  In particular, we may use a very similar prefactor setup and $U$ operator to examine possible congruences for the more general \begin{align}
D_{8k+7}(q) := \sum_{n=0}^{\infty}d_{8k+7}(n)q^n = \prod_{m=1}^{\infty}\frac{(1-q^{2m})^{8k+7}}{(1-q^m)^{24k+22}}\label{D8k7}.
\end{align}  It is tempting to believe that more general congruence families exist in this form, and experimental evidence appears at first sight to favor the existence of other congruence families.  However, we have thus far carefully searched for such a congruence family in vain.  Certainly, individual congruences exist for $d_k(n)$ over large powers of any prime $\ell$, provided that one varies $k$.  For example, Baruah, Das, and Talukdar \cite{Baruahetal} have proved that if \begin{align}
d_k(2^Mn+r)\equiv 0\pmod{2^N},\label{baruahext1}
\end{align} then for any $j\ge 0$, \begin{align}
d_{k+2^{M+N-1}j}(2^Mn+r)\equiv 0\pmod{2^N}.\label{baruahext2}
\end{align}  We obviously have an infinite family corresponding to (\ref{baruahext1}).  However, notice that this parallel set of congruences (\ref{baruahext2}) does not necessarily correspond to an infinite family for a fixed $k$; as the power of 2 increases, so too does the index $k_0$ of $d_{k_0}(n)$.

On the other hand, such a proliferance of individual congruences should tempt one to search for other congruence families for different but fixed $k$ that somehow run parallel to those of the family proved in this article.  Nevertheless, we have been unable to find such a family, despite a long search.

Indeed, we have reason to believe that such families are actually quite rare, and indeed outright nonexistent, for most values of $k$.

We are in the process of conducting a detailed investigation of when such families are able to occur; however, it certainly demands a much more thorough approach than is appropriate for the remainder of this paper.

We also note that some other congruence families have been found for $d_k(n)$ modulo powers of 3 \cite{AndrewsPaule}, \cite{Smoot2}, and powers of 5 \cite{Banerjee}.  Congruence families for other primes appear very likely.  In both of the latter cases, the classical techniques (that is, manipulating the associated functions $L_{\alpha}$ as polynomials in a given Hauptmodul) are inadequate, and the more recent techniques of localization appear necessary.  Indeed, part of what makes Theorem \ref{Thm12} so remarkable is the comparable simplicity of its proof, which only depends on classical methods.

We expect that a more complete understanding of the circumstances in which an infinite family of congruences modulo powers of any prime $\ell$ exists for $d_k(n)$ (and equally important, the circumstances in which such a family does \textit{not} exist) will be possible very soon.

\section{Appendix}

\begin{align}
U\left( 1 \right) &= 3640 x + 6231264 x^2 + 2312134656 x^3 + 355521773568 x^4 + 
 29519893757952 x^5\label{relationapu1}\\ &+ 1521646554841088 x^6 + 52926661334663168 x^7 + 
 1310564837327110144 x^8\nonumber\\ &+ 23942486575294709760 x^9 + 
 330518389127517306880 x^{10} + 3501373210632581545984 x^{11}\nonumber\\ &+ 
 28712526880139873615872 x^{12} + 182757766232154618986496 x^{13} + 
 899971247072045128744960 x^{14}\nonumber\\ &+ 3393979548633672982724608 x^{15} + 
 9614627478658155397775360 x^{16}\nonumber\\ &+ 19791437929706683090599936 x^{17} + 
 27937520112656821084225536 x^{18}\nonumber\\ &+ 24178516392292583494123520 x^{19} + 
 9671406556917033397649408 x^{20},\nonumber\\
U\left( x \right) &= 480 x + 2451840 x^2 + 1830703104 x^3 + 495729844224 x^4 + 
 68082221187072 x^5\label{relationapux}\\ &+ 5619929991610368 x^6 + 308140965409849344 x^7 + 
 11967581531943206912 x^8\nonumber\\ &+ 343933610805310783488 x^9 + 
 7541976859963962687488 x^{10} + 128987018531625607626752 x^{11}\nonumber\\ &+ 
 1747438683785531255422976 x^{12} + 18952626245617873030479872 x^{13}\nonumber\\ &+ 
 165656328233775543590322176 x^{14} + 
 1170336319370329141589049344 x^{15}\nonumber\\ &+ 
 6679520161261089361353506816 x^{16} + 
 30665365006400143347928793088 x^{17}\nonumber\\ &+ 
 112261430099308616694754181120 x^{18} + 
 322997173707177779010471985152 x^{19}\nonumber\\ &+ 
 713935978476697717639431061504 x^{20} + 
 1169389109609952342176585220096 x^{21}\nonumber\\ &+ 
 1336356272408568006753604599808 x^{22} + 
 950737950171172051122527404032 x^{23}\nonumber\\ &+ 
 316912650057057350374175801344 x^{24},\nonumber\\
U\left( x^2 \right) &= 34 x + 649128 x^2 + 999188736 x^3 + 469960929792 x^4 + 
 103575916707840 x^5\label{relationapux2}\\ &+ 13142823130628096 x^6 + 
 1080778710061154304 x^7 + 62098461599056527360 x^8\nonumber\\ &+ 
 2623660269082228293632 x^9 + 84519653872764356395008 x^{10} + 
 2131706476697046604578816 x^{11}\nonumber\\ &+ 42929493546501248057868288 x^{12} + 
 700513558478561481590833152 x^{13}\nonumber\\ &+ 
 9363134200724858517753692160 x^{14} + 
 103309973951373480156995780608 x^{15}\nonumber\\ &+ 
 945835294599965780764891545600 x^{16} + 
 7205353092856347021622577725440 x^{17}\nonumber\\ &+ 
 45697911126683724742728645345280 x^{18} + 
 240861503115473072724796124430336 x^{19}\nonumber\\ &+ 
 1050474963615392708649645823229952 x^{20} + 
 3763121389405690915490299019001856 x^{21}\nonumber\\ &+ 
 10948761369113220895121045440167936 x^{22} + 
 25446382394087648551788055652990976 x^{23}\nonumber\\ &+ 
 46096489899968019766506944412516352 x^{24} + 
 62687857482486400362814966912253952 x^{25}\nonumber\\ &+ 
 60177909294034506147851494565609472 x^{26} + 
 36346078009743793399713474304540672 x^{27}\nonumber\\ &+ 
 10384593717069655257060992658440192 x^{28},\nonumber
\end{align}
\begin{align}
U\left( x^3 \right) &= x + 116980 x^2 + 391517568 x^3 + 321019983616 x^4 + 
 111973236121600 x^5 + 21337425505910784 x^6\label{relationapux3}\\ &+ 
 2553730125959004160 x^7 + 209453764990082220032 x^8 + 
 12484823733419307433984 x^9\nonumber\\ & + 563791642924407796531200 x^{10} + 
 19886310763248535807721472 x^{11}\nonumber\\ & + 560619149820617402792017920 x^{12} + 
 12855690027118794791259734016 x^{13}\nonumber\\ & + 
 243068700423103102488194056192 x^{14} + 
 3829262655422818366954045177856 x^{15}\nonumber\\ & + 
 50666601737998691101231183560704 x^{16} + 
 566395672142771163690357451915264 x^{17}\nonumber\\ & + 
 5371596722455695258972289849360384 x^{18} + 
 43327315880729074970221060479254528 x^{19}\nonumber\\ & + 
 297518918103841878357244746758356992 x^{20} + 
 1738237469614080574950618042390806528 x^{21}\nonumber\\ & + 
 8621442084428477791752217394715557888 x^{22} + 
 36154467486544013569919867333480808448 x^{23}\nonumber\\ & + 
 127391720725152650083511609060455612416 x^{24}\nonumber\\ & + 
 373776737506583531887366397457143431168 x^{25}\nonumber\\ & + 
 901751212782819217578862421446940950528 x^{26}\nonumber\\ & + 
 1757278352654097795186046912413084680192 x^{27}\nonumber\\ & + 
 2696281874184257965081458786319057551360 x^{28}\nonumber\\ & + 
 3134485767560304742791290024023587553280 x^{29}\nonumber\\ & + 
 2593988433774263325971779478137092440064 x^{30}\nonumber\\ & + 
 1361129467683753853853498429727072845824 x^{31}\nonumber\\ & + 
 340282366920938463463374607431768211456 x^{32}.\nonumber
\end{align}

\section{Acknowledgments}
The second author was funded in whole by the Austrian Science Fund (FWF): Einzelprojekte P 33933, ``Partition Congruences by the Localization Method".  Our most profound thanks to the Austrian Government and People for their generous support.

\end{document}